\documentclass[12pt,a4paper]{amsart}

\usepackage[latin1]{inputenc}
\usepackage{amsmath,amsfonts,amssymb,setspace,amscd}
\usepackage{enumerate}
\usepackage[
      colorlinks=true,    
      urlcolor=blue,    
      menucolor=blue,    
      linkcolor=blue,    
      bookmarks=true,    
      bookmarksopen=true,    
      citecolor=blue,
      hyperfootnotes=false,    
      pdfpagemode=UseOutlines    
]{hyperref}
\usepackage[top=3 cm, bottom=3 cm, left=3 cm, right=3cm]{geometry}
\onehalfspace

\newtheorem{theorem}{Theorem}[section]
\newtheorem{lemma}[theorem]{Lemma}
\newtheorem{proposition}[theorem]{Proposition}

\newtheorem{corollary}[theorem]{Corollary}

\theoremstyle{definition}

\numberwithin{equation}{section}
\numberwithin{theorem}{section}

\newcommand{\R}{\mathbb{R}}
\newcommand{\Z}{\mathbb{Z}}
\newcommand{\Q}{\mathbb{Q}}
\newcommand{\C}{\mathbb{C}}
\newcommand{\Gm}{\mathbb{G}_m}

\newcommand{\Qbar}{\overline{\Q}}

\newcommand{\cC}{\mathcal{C}}

\newcommand{\ol}[1]{\overline{#1}}
\newcommand{\oa}{\overline{a}}

\newcommand{\tor}{\text{tor}}
\newcommand{\Kab}{K^{\mathrm{ab}}}

\def\lg{\left\lbrace}
\def\rg{\right\rbrace}

\author[F. Barroero]{Fabrizio Barroero}
\address{Universit\`a di Roma Tre, Dipartimento di Matematica e Fisica, Largo San Mu\-rial\-do 1, 00146 Roma, Italy}
\email{fbarroero@gmail.com}

\author[M. Sha]{Min Sha}
\address{Department of Computing, Macquarie University, Sydney, NSW 2109, Australia}
\email{shamin2010@gmail.com}

\title[Multiplicatively dependent torsion points]{Torsion points with multiplicatively dependent coordinates on elliptic curves}
\date{\today}

\subjclass[2010]{11G05, 11G50, 11U09, 14K05}   
\keywords{Elliptic curve, torsion point, multiplicative dependence, o-minimality}

\begin{document}

\begin{abstract}
In this paper, we study the finiteness problem of torsion points on an elliptic curve whose coordinates 
satisfy some multiplicative dependence relations. In particular, we prove that on an elliptic curve defined over a number field 
 there are only finitely many torsion points 
whose coordinates are multiplicatively dependent. Moreover, we produce an effective result when the elliptic curve is defined over the rational numbers or has complex multiplication.
\end{abstract}

\maketitle

\section{Introduction}

\subsection{Background and motivation}

Let  $\Gm$ be  the multiplicative algebraic group over the complex numbers $\C$. 
In \cite{BMZ99}, Bombieri, Masser and  Zannier initiated 
the study of intersections of geometrically irreducible  algebraic  curves $\cC \subset \Gm^n,n \ge 2$, defined over a number field $K$, and  a union of proper algebraic subgroups of $\Gm^n$. 
It is well known that  each such subgroup of $\Gm^n$ is defined by a finite set of equations of the shape $X_1^{k_1}\cdots X_n^{k_n}=1$, with integer exponents not all zero. 
That is, the work~\cite{BMZ99} is about multiplicative dependence of points on a curve. 
They proved, that under the assumption that $\cC$ is not contained in
any translate of a proper algebraic subgroup of $\Gm^n$, the points on $\cC(\overline{\Q})$ 
with multiplicatively dependent coordinates form a set of bounded (absolute logarithmic) Weil  height. 

More recently, a new point of view was introduced in~\cite{OSSZa} 
by establishing  the structure of   points with multiplicatively dependent coordinates in 
$\cC(\Kab)$,  where $\Kab$ is the maximal abelian extension of $K$. 
In turn this implies that the set of such points is finite if $\cC$ is of positive genus, 
which has already been extended to the genus zero case \cite{OSSZb}. 
In particular, for an elliptic curve $E$ defined over $K$ with complex multiplication (CM), 
since the field generated over $K$ by all the torsion points of $E$ is contained in $\Kab$, 
it follows that there are only finitely many torsion points of $E$ with multiplicatively dependent coordinates. By this, here and in the rest of the paper, we mean that we fix a model of $E$ in $\mathbb{P}^2$, for instance the Weiertrass model, and take affine coordinates $(X,Y)$ so that the origin of $E$ becomes the point at infinity.

In this paper, we want to study the finiteness problem of torsion points on $E$ 
whose coordinates are multiplicatively dependent by dropping the CM condition. 

It is not hard to find elliptic curves that have some torsion points whose coordinates are multiplicatively dependent. 
For instance, the point $(1,1)$ has order four on the elliptic curve of equation $Y^2=X^3-X^2+X$. 
Moreover, for the curve $E$ defined by $Y^2 = X(X - 1)(X - \zeta^2)$, where $\zeta \ne \pm 1$ is a root of unity, 
by \cite[Proposition 4.3]{Stoll}, 
 $(\zeta, \sqrt{\zeta(\zeta-1)(\zeta-\zeta^2)})$ is a torsion point on $E$, 
 whose coordinates are certainly multiplicatively dependent.

\subsection{Main results and methods} 
\label{sec:intro}

It is in fact a consequence of the Zilber-Pink Conjectures (see, e.g., \cite{Pink}) that, for a fixed elliptic curve over the complex numbers, there are at most finitely many torsion points with multiplicatively dependent coordinates.

In this paper, we confirm this fact for elliptic curves over the algebraic numbers, and we provide an effective proof in some special cases.

The following theorem is our first result. Note that it gives a special case of a conjecture of Pink (Conjecture 5.1 in \cite{Pink}).

\begin{theorem}\label{mainthm}
	Let $E$ be an elliptic curve defined over a number field $K$, and let $n \ge 1$ be an integer. 
	Let $\cC$ be an irreducible curve in $E\times \Gm^n$, also defined over $K$, with coordinates $(X,Y, Z_1, \dots , Z_n)$ such that $(X,Y)$ is not a torsion point of $E$ and  $Z_1, \dots , Z_n$ are multiplicatively independent. Then, there are at most finitely many  points $c\in \cC(\C)$ such that $(X(c),Y(c))$ is a torsion point of $E$ and  $Z_1(c), \dots , Z_n(c)$ are multiplicatively dependent.
\end{theorem}

If $n=1$ the statement is a special case of the Manin-Mumford conjecture, proved for semi-abelian varieties by Hindry \cite{Hindry88}.
In case $n=2 $ and $Z_1=X$ and $Z_2=Y$, we deduce that on an elliptic curve over a number field there are at most finitely many torsion points whose coordinates are multiplicatively dependent. More generally, we have:

\begin{corollary}\label{cor}
Let  $E$ be an elliptic curve defined over a number field $K$. 
Let $f_1, \dots, f_n \in K(X,Y)$ be multiplicatively independent non-zero rational functions in two variables with coefficients in $K$. 
Then, there are only finitely many torsion points $(x,y)$ on $E$ such that $f_1(x,y), \dots , f_n(x,y)$ are multiplicatively dependent.
\end{corollary}

The proof of the above theorem can be obtained by slightly modifying the argument used by the first author with Capuano in \cite{BC2017} to prove the following result.

\begin{theorem}[\cite{BC2017}, Theorem 1.2] \label{ThmBC2017}
	Let $\cC\subseteq \mathbb{A}^{2m+1}\times \mathbb{G}_\emph{m}^n$ be an irreducible curve defined over $\Qbar$ with coordinate functions $(X_1,Y_1,\dots ,X_{m}, Y_{m}, \lambda,Z_1, \dots , Z_{n} )$, $\lambda$ non-constant, such that, for every $i=1,\dots , m$, the points $P_i=(X_i,Y_i)$ lie on the Legendre elliptic curve $E_\lambda$ of equation $Y^2=X(X-1)(X-\lambda)$. Suppose moreover that no non-trivial relation among $P_1,\dots , P_m$ holds and that the $Z_1, \dots , Z_{n}$ are multiplicatively independent.
	Then, there are at most finitely many ${c}\in \cC(\C)$ such that there exist $(a_1,\dots, a_m) \in \Z^m\setminus \{0 \}$ and $( b_1, \dots , b_n) \in \Z^n \setminus \{0 \}$ for which
	\begin{equation*}\label{rel2}
	a_1P_1({c})+\dots +a_mP_m ({c})= O \text{  and     } Z_1({c})^{b_1}\cdots Z_n({c})^{b_n}=1 .
	\end{equation*}
\end{theorem}

Note that the above statement implies that, if $P=(X,Y)$ is a point in $E_{\lambda}\left(\overline{\Q(\lambda)}\right)$ that has infinite order, there are at most finitely many $\lambda_0 \in \C$ such that $P(\lambda_0) $ is torsion on $E_{\lambda_0}$ and $X(\lambda_0)$ and $Y(\lambda_0)$ are multiplicatively dependent.

We would like to point out that, although in Theorem \ref{ThmBC2017} only the Legendre family is considered, in the same paper \cite{BC2017} the authors formulate a theorem (Theorem 1.3 there) that deals with general elliptic schemes, so one is not restricting to the Legendre model only.

As mentioned in \cite{BC2017}, the proof  of Theorem \ref{ThmBC2017} goes through when one has a constant elliptic curve over the algebraic numbers, unless $(P_1,\dots , P_m)$ and $(Z_1, \dots, Z_n)$ are both contained in a non-torsion translate of an abelian subvariety of $E^m$ and a subtorus of $\Gm^n$, respectively. Indeed, Silverman's bounded height Theorem \cite{Silverman}, a fundamental ingredient in the proof, requires $\lambda$ not to be constant, but a result of Bombieri, Masser and Zannier \cite{BMZ99} gives boundedness of the height in case the $Z_j$ are independent modulo constants, while Viada \cite{Viada2003} proved the analogous result for a constant elliptic curve $E$ defined over the algebraic numbers.

Under the hypothesis of Theorem \ref{mainthm}, one obtains the required bound on the height simply by the fact that, on the elliptic curve side, we are considering torsion points, which have height bounded by an absolute constant depending only on the elliptic curve.

In Section \ref{sketch} we give a proof of Theorem \ref{mainthm} following the usual Pila-Zannier strategy.

We now turn our attention to the issue of effectivity. 

The only point of the proof of Theorem \ref{mainthm} that is not clearly effective ultimately lies in the point-counting part.  It is likely that the recent work of Jones and Thomas \cite{JT18} or the one of Binyamini \cite{Bin} would give effective Pila-Wilkie type estimates for the number of rational points of bounded height on the subanalytic surface that we consider. As we deal with multiplicative dependencies and not only torsion points, this is not sufficient and we have to invoke a result of Habegger and Pila (Corollary 7.2 of \cite{HabPila16}). Their technique is quite sophisticated and, as far as we know, the possibility of extending the above mentioned effective results to this context has not been worked out yet.

In this paper, we introduce a new method to obtain an effective version of Theorem~\ref{mainthm} 
when the coordinate functions $Z_1, \dots , Z_n$ are rational functions in $X$ and $Y$ and coefficients in $K$, 
which can be used whenever we have a lower bound for the height on $K(E_\tor)^*\setminus \mu_\infty$, i.e., the field $K(E_\tor)$ has the so-called Bogomolov Property. 

Here, $E_\tor$ is the set of $\Qbar$-torsion points on $E$, $K(E_\tor)$ is the field generated over $K$ by all the coordinates of the points in $E_\tor$, 
and $\mu_\infty$ denotes the set of roots of unity in $\C$. 

\begin{theorem}\label{effthm}
Let $E$ be an elliptic curve defined over a number field $K$. 
Let $f_1, \dots, f_n \in K(X,Y)$ be multiplicatively independent non-zero rational functions in two variables with coefficients in $K$. Suppose there exists a positive constant $\epsilon$ such that $h(\alpha)\geq \epsilon$ for all $\alpha \in K(E_\tor)^*\setminus \mu_\infty$. Then, there is an effectively computable constant $C=C(E, f_1, \dots , f_n, \epsilon)$ such that, if $(x,y)$ is a torsion point of $E$ such that $f_1(x,y), \dots , f_n(x,y)$ are multiplicatively dependent, then $(x,y)$ has order at most $C$.
\end{theorem}

Therefore, if we have an effective lower bound for the height on $K(E_\tor)^*\setminus \mu_\infty$ in terms of $K$ and $E$ we get an effective constant $C(E, f_1, \dots , f_n)$ that bounds the order of a torsion $(x,y)$ on $E$ such that $f_1(x,y), \dots , f_n(x,y)$ are multiplicatively dependent.

There are two cases for which we have this positive effective lower bound.

First, it is well-known that, if $E$ has complex multiplication, the field extension $K(E_\tor)/K$ is abelian and, therefore, thanks to a result of Amoroso and Zannier \cite{AZ10}, we have the desired effective lower bound, that actually only depends on the degree of $K$ over $\Q$.

In case $E$ is defined over the rational numbers and has no complex multiplication, Frey \cite{Frey18} has recently effectivized and extended the work \cite{Hab13} of Habegger. To be more specific, in Theorem 1.2 of \cite{Frey18}  she gave an explicit lower bound for the height on $K(E_\tor)^*\setminus \mu_\infty$ that depends on a prime number that is supersingular and surjective for $E$ and large enough with respect to the degree of the Galois closure of $K$ over $\Q$. Explicit bounds for a supersingular and surjective prime (but still large enough) for $E$ can be found combining Theorem 2.7 in \cite{Frey17} and Theorem 5.2 of \cite{LF16}.

We can then formulate the following result.

\begin{corollary}
	Let $E$ be an elliptic curve defined over $\Q$ or an elliptic curve with complex multiplication. There is an effectively computable constant $C=C(E)$ such that, if $(x,y)$ is a torsion point of $E$ such that $x$ and $y$ are multiplicatively dependent, then $(x,y)$ has order at most $C$.
\end{corollary}

\section{Proof of Theorem \ref{mainthm}} \label{sketch}

As mentioned in the Introduction, the proof of Theorem \ref{mainthm} can be obtained by adapting the proof of Theorem \ref{ThmBC2017} (Theorem 1.2 in \cite{BC2017}) to this setting. In particular, one follows the general strategy introduced by Pila and Zannier in \cite{PilaZannier} using the theory of o-minimal structures to give an alternative proof of the Manin-Mumford conjecture for abelian varieties.
An important ingredient of the proof is the now well-known Pila-Wilkie Theorem \cite{PilaWilkie}, which provides an estimate for the number of rational points on a ``sufficiently transcendental'' real subanalytic variety. These rational points correspond to torsion points. For more details about the general strategy and how it has been applied to other problems we refer to \cite{Zannier}.

On the other hand, if one wants to deal with points lying in proper algebraic subgroups like in Theorem \ref{mainthm}, a more refined result is needed. For instance, first in \cite{BC2016} and then in \cite{BC2017}, ideas introduced in \cite{CMPZ} were adapted to deal with linear relations rather than just with the torsion points.

We now turn to the proof of Theorem \ref{mainthm}.

We call $\cC_0$ the set of points of $\cC$ we want to prove to be finite.
First, we notice that the points in $\cC_0$ must be algebraic and, since their projection on $E$ is a torsion point, they have bounded height. We then just have to exhibit a bound on their degree over the number field $K$.

\begin{lemma} \label{lem:Galois}
	There exists a compact (in the complex topology) subset $\cC^*$ of $\cC$, such that for all $c\in \cC_0$ of degree large enough, 
	at least half of the Galois conjugates of $c$ over $K$ lies in $\cC^*$
\end{lemma}

\begin{proof}
	See Lemma 8.2 of \cite{MZ14a}.
\end{proof}

Note that, if $c\in \cC_0$, then all its Galois conjugates over $K$ must also lie in $\cC_0$.

We now cover $\cC^*$ with finitely many discs $D_1, \dots , D_{\gamma_1}$.

Let $D$ be one of these discs. For $a\in \Z \setminus \{0\}$ and $\ol{b}=(b_1,\dots , b_n )\in \Z^n \setminus \{0\}$ we call $D(a,\ol{b})$ the set of $c\in D$ such that $(X(c),Y(c))$ has order dividing $a$ and $\prod Z_j(c)^{b_j}=1$.

For the rest of the section the implied constants will depend on $\cC$, $E$ and $K$. Any further dependence will be expressed by an index.

\begin{lemma}
If $c\in D \cap \cC_0$, there are $a\in \Z \setminus \{0\}$ and $\ol{b}\in \Z^n \setminus \{0\}$ such that $c \in D(a,\ol{b})$ and
\begin{equation}   \label{eq:abc}
\max \{|a|,|\ol{b}| \} \ll [K(c):K]^{\gamma_2}
\end{equation}
for some fixed $\gamma_2>0$, where $|\ol{b}| = \max \{|b_1| , \dots , |b_n|\}$.
\end{lemma}

\begin{proof}
We proceed as done  in Lemmas 5.1 and 5.2 of \cite{BC2017}.

First, one can use the work of David \cite{David97} to bound the order of a torsion point of $E$ in terms of the degree of a field of definition.

By Theorem $\mathbb{G}_m$ of Masser \cite{Masser88} we take $\ol{b}$ to be such that 
$$
|\ol{b}|\leq n^{n-1} \omega \left( \frac{h}{\eta}  \right)^{n-1},
$$
where $\omega$ is the number of roots of unity in $K(c)$, which is clearly polynomially bounded in terms of $[K(c):\Q]$, $h$ is an upper bound for the height of the values $Z_j(c)$, which is again bounded because $h(c)$ is bounded, and $\eta$ is a lower bound on the height for elements of $K(c)^*\setminus \mu_\infty$. For the latter, we could use the celebrated result of Dobrowolski but the weaker and older work of Blanksby and Montgomery \cite{BlaMon} suffices here.
\end{proof}

We consider an elliptic logarithm $u$ of $(X,Y)$ and the principal determinations of the standard logarithms $w_1, \dots , w_n$ of $Z_1, \dots , Z_n$ seen as analytic functions on (an open neighbourhood of) $D$ and the equations
$$
u=p+q \tau,  \quad   w_j=r_j + 2\pi i  s_j , \ \  \mbox{for $j=1,\dots , n$,}
$$
where $(1,\tau)$ is a basis of the period lattice of $E$. If we consider the real coordinates $p, q, r_j, s_j$ as a function
$$
\begin{array}{crcl}
\theta :& D\subset \R^2 & \rightarrow & \R^{2+2n} \\
& c &\mapsto & (p(c), q(c), r_1(c), s_1(c),\dots , r_n(c),s_n(c))
\end{array}
$$
of a local uniformizer on the compact disc $D$ seen as a subset of $\R^2$, the image $\theta(D)$ is a subanalytic surface $S$. Note that $\theta$ is injective. Moreover,  note that, as $D$ is compact, we have that $p,q$ and the $s_j$ take bounded values.
The $p, q, r_1, s_1,\dots , r_n,s_n$ are sometimes called Betti-coordinates and $\theta$ Betti-map.

The points of $\cC_0$ that yield two relations will correspond to points of $S$ lying on linear varieties defined by equations of some special form and with integer coefficients.
In particular, if  $c\in D(a,\ol{b})$, there are integers $e,f,g$ such that
$$
\begin{array}{c}
ap=e, \\
aq=f, \\
b_1r_1+ \dots +b_n r_n =0, \\
b_1s_1+ \dots +b_n s_n =g,
\end{array}
$$
hold for the image $\theta(c)$.

We define
\begin{multline*}
W=\{(\alpha,\beta_1, \dots , \beta_n, \sigma_1, \sigma_2, \sigma_3, p, q, r_1, s_1,\dots , r_n,s_n )\in \R^{n+1}\times \R^3\times S: \\ \alpha p= \sigma_1,  \alpha q= \sigma_2, \sum_{j=1}^n \beta_j r_j=0,  \sum_{j=1}^n \beta_j s_j=\sigma_3  \},
\end{multline*}
and, for $\alpha,\beta_1, \dots , \beta_n \in \R$ and $\ol{\beta}=(\beta_1, \dots , \beta_n)$, the fiber
\begin{multline*}
W_{\alpha,\ol{\beta}}=\{ (\sigma_1, \sigma_2, \sigma_3, p, q, r_1, s_1,\dots , r_n,s_n )\in  \R^3\times S:  \\(\alpha,\beta_1, \dots , \beta_n, \sigma_1, \sigma_2, \sigma_3, p, q, r_1, s_1,\dots , r_n,s_n  ) \in W \}.
\end{multline*}
We let $\pi_1$ be the projection from $\R^3\times S \subseteq \R^3 \times \R^{2n+2}$ to $\R_3$, while $\pi_2$ indicates the projection to $S$. We also define, for $T\geq 0$
\begin{multline*}
W_{\alpha,\ol{\beta}}^\sim (\Q, T)= \{ (\sigma_1, \sigma_2, \sigma_3, p, q, r_1, s_1,\dots , r_n,s_n  ) \in W_{\alpha,\ol{\beta}}: \\ (\sigma_1,\sigma_2,\sigma_3) \in \Q^3 \text{ and } H(\sigma_1,\sigma_2,\sigma_3)\leq T   \},
\end{multline*}
where $H(\sigma_1,\sigma_2,\sigma_3)$ is the maximum of the absolute values of the numerators and denominators of the $\sigma_j$ when they are written in lowest terms.

Fix now $a$ and $\ol{b}$.
Note that, if $c\in D(a, \ol{b})$, there are integers $e,f,g$ such that $(e,f,g, \theta (c))\in W_{a,\ol{b}}$. Since $p,q,s_1, \dots, s_n$ take bounded values as $D$ is a compact disc, we can suppose that 
$$
\max \{|e|,|f|,|g|,|a|,|\ol{b}| \}\leq T_0
$$ 
for some $T_0$ with $T_0 \ll \max \{|a|,|\ol{b}| \} $. 
Therefore, if we let 
$$
\Sigma_{a,\ol{b}} :=\pi_2^{-1} (\theta(D(a,\ol{b})))\cap W_{a,\ol{b}},
$$
we have $\Sigma_{a,\ol{b}}  \subseteq   W_{a,\ol{b}}^\sim (\Q, T_0) $.

We claim that, for any $\epsilon>0$, we have an upper bound of the form 
\begin{equation}   \label{eq:ab}
|D(a,\ol{b})|\ll_\epsilon (\max \{|a|,|\ol{b}|\})^\epsilon.
\end{equation}
If not, by the previous considerations the following lemma would be contradicted.

\begin{lemma}
	For any $\epsilon>0$ we have $|\pi_2(\Sigma_{a,\ol{b}})|\ll_\epsilon T_0^\epsilon$. 
\end{lemma}

\begin{proof}
	Suppose there is a positive constant $\gamma_3=\gamma_3(W,\epsilon)$ such that $|\pi_2(\Sigma_{a,\ol{b}})|\geq \gamma_3 T_0^\epsilon$. Then, by Corollary 7.2 of \cite{HabPila16}, there exists a definable function $\delta:[0,1]\rightarrow W_{a,\ol{b}}$ such that
	\begin{enumerate}
		\item the map $\delta_1:=\pi_1 \circ \delta :[0,1]\rightarrow \R^3$ is semi-algebraic and its restriction to $(0,1)$ is real analytic;
		\item the composition $\delta_2:=\pi_2 \circ \delta :[0,1]\rightarrow S$ is non-constant;
		\item we have $\pi_2(\delta(0)) \in \pi_2(\Sigma_{a,\ol{b}})$.
	\end{enumerate}
	
	By rescaling and restricting the domain we can suppose that the path $\delta_1 $ is contained in a real algebraic curve. Moreover, by (3) above, there is a $c_0 \in D(a,\ol{b})$ with $\theta (c_0)=\delta_2(0)$.
	
	We now consider the map
	$$
	\begin{array}{clcl}
	\phi :& E\times \Gm^n & \rightarrow & E\times\Gm \\
	& (X,Y,Z_1, \dots , Z_n) &\mapsto & (X,Y,\prod_{j=1}^n Z_j^{b_j})
	\end{array}
	$$
	and its differential
	$$
	\begin{array}{clcl}
	d \phi :& \C\times \C^n & \rightarrow & \C\times\C \\
	& (u,w_1, \dots , w_n) &\mapsto & (u,w':=\sum_{j=1}^n b_j w_j). 
	\end{array}
	$$
	We now see $\sigma_1,\sigma_2, \sigma_3, u, w'$ as coordinate functions on $[0,1]$. 
	We have that the transcendence degree trdeg$_\C \, \C(\sigma_1,\sigma_2,\sigma_3)\leq 1$ and recall the two relations $au=\sigma_1+\sigma_2 \tau$ and $w'= 2\pi i \sigma_3$. We deduce that trdeg$_\C \, \C(\sigma_1,\sigma_2,\sigma_3,u,w')\leq 1$. This gives a map $\delta':=(u,w'): [0,1]\rightarrow \C \times \C$ that is real semi-algebraic, continuous and with $\delta'  |_{(0,1)}$ real analytic. By Ax's Theorem \cite{Ax72} (see \cite[Theorem 5.4]{HabPila16}), the Zariski closure in $E\times\Gm$ of the image of $\exp \circ \delta'$, which is contained in $\phi(\cC)$, is a coset, that must actually be a torsion coset because $\phi (c_0)$ is a torsion point of $E\times\Gm$. Now, since $u$ is non-constant, this coset must be a curve which then coincides with $\phi(\cC)$. This contradicts the hypotheses of Theorem~\ref{mainthm}.
\end{proof}

Finally, for a $c\in \cC_0$ of large degree over $K$, by Lemma~\ref{lem:Galois}  we have that one of the discs $D_1,\dots, D_{\gamma_1}$, say $D_1$, 
contains at least $ [K(c):K]/(2\gamma_1)$ conjugates of $c$. Moreover, if $c\in D_1(a,\ol{b})$ for some $a$ and $\ol{b}$, all of these conjugates belong to $D_1(a,\ol{b})$. Therefore, combining this with \eqref{eq:abc} and \eqref{eq:ab}, we get
$$
[K(c):K] \ll |D_1(a,\ol{b})|\ll_\epsilon (\max \{|a|,|\ol{b}|\})^\epsilon \ll [K(c):K]^{\gamma_2 \epsilon}, 
$$
which, after choosing $\epsilon < 1/(2\gamma_2)  $, leads to a contradiction if $[K(c):K]$ is too large.
This completes the proof of Theorem~\ref{mainthm}.

\section{Proof of Theorem \ref{effthm}}

We start by introducing a couple of results that we are going to use in proving Theorem \ref{effthm}.
The first is a results of Masser \cite{Masser88}.

Let $F$ be a real-valued non-constant function on $\Z^n$ that satisfies the following conditions:

\begin{enumerate}
	\item  $F(\oa)\geq 0$ for all $\oa \in \Z^n$;
	\item $F(\lambda \oa)=|\lambda|F(\oa)$ for all $\oa \in \Z^n$ and all $\lambda \in \Z$;
	\item $F(\oa_1+\oa_2)\leq F(\oa_1)+ F(\oa_2)$ for all $\oa_1, \oa_2 \in \Z^n$.
 \end{enumerate}
The set $\Gamma=\{\oa \in \Z^n : F(\oa)=0 \}$ is a lattice in $\R^n$.

\begin{proposition}[\cite{Masser88}, Proposition on p. 250]\label{propMasser}
	Let $F$ and $\Gamma$ be as above. Suppose there exists some real number $B$ such that $F(e_i)\leq B$, where the $e_i$ are the elements of the standard basis of $\R^n$. Suppose moreover that there is a positive real number $\epsilon$ such that $F(\oa)\geq \epsilon$ for all $\oa \not \in \Gamma $. Then, $\Gamma$ contains a non-zero element $\oa$ with
	$$
	|\oa| \leq  (n B/\epsilon)^{n-1},
	$$
where $|\ol{a}|$ is the maximal absolute value of the coordinates of $\ol{a}$.
\end{proposition}

We now recall an effective version of the Manin-Mumford conjecture for semiabelian varieties due to Hrushovski \cite{Hrushovski01}. We only state a much weaker version.

\begin{theorem}\label{MM}	Let $E$ be an elliptic curve and $\cC$ be a curve in $E\times \Gm$. Suppose $E$ and $\cC$ are defined over a number field $K$ and that $\cC$ is not of the form $\{T\}\times \Gm$ or $E\times \{\zeta\}$ for some $T\in E_\tor$ or $\zeta \in \mu_\infty$. Then, there is an effectively computable constant $N=N(\cC,K)$ such that $\cC$ contains at most $N$ torsion points of $E\times \Gm$.
\end{theorem}

We recall that $E_\tor$ is the set of $\Qbar$-torsion points of $E$.

We start by applying Proposition \ref{propMasser} to our setting. 
Consider a point $(x,y)\in E_\tor$ such that $f_1(x,y), \dots , f_n(x,y)$ are multiplicatively dependent. We are excluding the finitely many $(x,y)$ where some of the $f_i$ vanish. We are going to show that there is a ``small'' vector $\oa$ such that $\prod_i f_i(x,y)^{a_i}$ is a root of unity.
 
 We define  
 \begin{equation*}\begin{array}{clcl}
 F: & \Z^n  &\rightarrow & \R_{\geq 0} \\
 &  (a_1, \dots, a_n) &\mapsto & h(\prod_i f_i(x,y)^{a_i}), 
 \end{array}
 \end{equation*}
 where $h$ denotes the absolute logarithmic Weil height. 
 By the basic properties of heights, $F$ satisfies the conditions that allow us to apply Masser's Proposition \ref{propMasser}. Moreover, 
 $$
 \Gamma=\lg\oa \in \Z^n : F(\oa)=0 \rg=\lg  \oa \in \Z^n: \prod_i f_i(x,y)^{a_i} \in \mu_\infty \rg .
 $$
 
Note that, since $(x,y)$ is a torsion point, we have that, for all $i=1,\dots, n$, $h(f_i(x,y))$ is effectively bounded from above by a constant $B=B(E,f_1,\dots ,f_n)$. Thus, $F(e_i)\leq B$ for the standard basis vectors $e_i$. On the other hand, any product $\prod_i f_i(x,y)^{b_i}$ lies in $K(E_\tor)$ and by our hypothesis we know there exists an effective constant $\epsilon$ such that $h(\prod_i f_i(x,y)^{b_i})\geq \epsilon$ in case $\prod_i f_i(x,y)^{b_i}$ is not a root of unity.

Proposition \ref{propMasser} tells us that, in case the $f_i(x,y)$ are multiplicatively dependent, there is a non-zero $\oa \in \Z^n$ with $|\oa| \leq  (n B/\epsilon)^{n-1}$, such that $\prod_i f_i(x,y)^{a_i} \in \mu_\infty $.

Now, for a fixed $\ol{b}\neq 0$, after possibly removing a divisor from $E$, we can define a morphism $E\rightarrow E\times \Gm $, $(X,Y)\mapsto \left((X,Y),\prod_i f_i(X,Y)^{b_i}\right) $ and let $C_{\ol{b}}$ be its image. This is a curve in $G:=E\times \Gm$.
Since the $f_i$ are not identically multiplicatively dependent, all of these curves, for varying $\ol{b}\neq 0$, satisfy the hypotheses of Theorem \ref{MM}. Therefore, there is an effective bound $N(C_{\ol{b}}, K) $ on the cardinality of $C_{\ol{b}}\cap G_\tor$.

Let us go back to our point $(x,y)\in E_\tor$ such that $f_1(x,y), \dots , f_n(x,y)$ are multiplicatively dependent. This will correspond to a point $((x,y),\zeta) \in C_{\ol{a}}\cap G_\tor$ for some $\oa \in \Z^n \setminus\{0\}$, where, by the above considerations, we can suppose $|\oa| \leq  (n B/\epsilon)^{n-1}$.

Therefore, the finite sum
\begin{equation} \label{eq1}
\sum_{ \substack{ \oa \in \Z^n\setminus \{0\} \\ |\oa| \leq  (n B/\epsilon)^{n-1} } }  N(C_{\ol{a}}, K)=: A(E, f_1, \dots , f_n, \epsilon)
\end{equation}
is an effective bound on the number of $(x,y)\in E_\tor$ such that $f_1(x,y), \dots , f_n(x,y)$ are multiplicatively dependent.

Now, given a point $(x,y)\in E_\tor$, any Galois conjugate $(x^\sigma,y^\sigma)$ of $(x,y)$ over $K$ is again a torsion point such that the $f_i(x^\sigma,y^\sigma)$ are multiplicatively dependent. Therefore, the constant $A(E, f_1, \dots , f_n, \epsilon)$ defined in \eqref{eq1} gives an effective upper bound on the degree of such points over $K$ and therefore an upper bound on their order as torsion points. 
This completes the proof.

\section*{Acknowledgement}

The authors are grateful to Gabriel Dill, Linda Frey, Philipp Habegger, Igor Shparlinski and Umberto Zannier for helpful discussions. 

They moreover would like to thank the anonymous  referee for comments and suggestions that greatly improved this article.

The first author has done part of this work under the support of the Swiss National Science Foundation [165525].
The second author is supported by a Macquarie University Research Fellowship.


\end{document}